\documentclass{article}

\usepackage[utf8]{inputenc}
\usepackage{amsfonts}
\usepackage{amsmath}
\usepackage[english]{babel}
\usepackage{amsthm}
\usepackage[symbol]{footmisc}
\usepackage{comment}
\usepackage{array}
\usepackage{color}
\usepackage{cancel}
\usepackage{bm}
\usepackage{graphicx} 
\usepackage{xcolor}

\newtheorem{theorem}{Theorem}[section]
\newtheorem{proposition}[theorem]{Proposition}
\newtheorem{lemma}[theorem]{Lemma}
\newtheorem{corollary}[theorem]{Corollary}

\newtheorem{definition}{Definition}[section]
\newtheorem{remark}{Remark}

\title{Iterates of composition operators  on global spaces of ultradifferentiable
functions}
\author{Héctor Ariza, Carmen Fernández and Antonio Galbis}

\begin{document}
\maketitle	

\begin{abstract}
We analyze the behavior of the iterates of composition operators defined by polynomials acting on global classes of ultradifferentiable functions of Beurling type and being invariant under Fourier transform. We characterize the polynomials $\psi$ for which the sequence of iterates is equicontinuous between two different  Gelfand-Shilov spaces. For the particular case in which the weight $\omega$ is equivalent to a power of the logarithm, the result obtained characterizes the polynomials $\psi$ for which the composition operator $C_\psi$ is power bounded in ${\mathcal S}_\omega({\mathbb R}).$ Unlike the composition operators in Schwartz class, the Waelbroek spectrum of an operator $C_\psi$, being $\psi$ a polynomial of degree greater than one lacking fixed points is never compact. We focus on the problem of convergence of Neumann series. We deduce the continuity of the resolvent operator between two different Gelfand-Shilov classes for polynomials $\psi$ lacking fixed points. Concerning polynomials of second degree the most interesting case is the one in which the polynomial only has one fixed point: we provide some restrictions on the indices $d, d'$ that are necessary for the resolvent operator to be continuous between the Gelfand-Shilov classes $\Sigma_d$ and $\Sigma_{d'}.$
\end{abstract}

\section{Introduction}

The spaces ${\mathcal S}_\omega({\mathbb R})$  are versions of the Schwartz class in the ultradifferentiable setting as
they are invariant under the action of the Fourier transform. They were introduced in
\cite{bjorck}. The most relevant
cases correspond to the classical Gelfand-Shilov spaces $\Sigma_d({\mathbb R}).$
The study of these classes and of different operators acting on them is a very
active area of research. See for instance \cite{am,ajm,asensio_jornet,capiello,cprt,debrouwere}.

In this article we continue the research started in \cite{jmaa} about composition operators in Gelfand-Shilov spaces. In general, the Gelfand-Shilov classes are not composition invariant with a polynomial of degree greater than one (see for example \cite[Theorem 3.9]{jmaa}), however for every weight function $\omega$ that is subadditive a new weight $\sigma$ can be found so that $f\circ \psi\in {\mathcal S}_\sigma({\mathbb R})$ whenever $f\in {\mathcal S}_\omega({\mathbb R})$ and $\psi$ is a polynomial (\cite[Theorem 4.4]{jmaa}). Since the choice of $\sigma$ does not depend on the polynomial $\psi$ it makes sense to analyze the dynamics of the operator $C_\psi: {\mathcal S}_\omega({\mathbb R})\to {\mathcal S}_\sigma({\mathbb R}),\ \ f\mapsto f\circ\psi.$ Given a Fréchet space $X$, an operator $T:X\to X$ is said to be power bounded if the sequence of iterates $\left\{T^m:\ m\in {\mathbb N}\right\}$ is equicontinuous or, equivalently, $\left(T^m x\right)_m$ is a bounded sequence for each $x\in X.$ There has been much research regarding power bounded operators in different function spaces and their relationship with other dynamical properties such as the property of being mean ergodic. In Section \ref{sec:power} we characterize the polynomials $\psi$ for which the sequence of iterates is equicontinuous between two different (but related) Gelfand-Shilov spaces. For the particular case in which the weight $\omega$ is equivalent to a power of the logarithm, the result obtained characterizes the polynomials $\psi$ for which the composition operator $C_\psi$ is power bounded in ${\mathcal S}_\omega({\mathbb R}).$ In Section \ref{sec:spectrum} we begin by verifying that, unlike the composition operators in Schwartz class (\cite{advances}), the Waelbroek spectrum of an operator $C_\psi$, being $\psi$ a polynomial of degree greater than one lacking fixed points  is never compact. In this regard, it should be noted that an spectral theory of linear operators not necessarily everywhere defined on non-normable spaces can be found in \cite{ABR}. Then we focus on the problem of convergence of Neumann series, our main results being Proposition \ref{prop:resolvente-continua} and Theorem \ref{th:neumann-gevrey-secondorder}. In Proposition \ref{prop:resolvente-continua} we deduce the continuity of the resolvent operator between two different Gelfand-Shilov classes for polynomials $\psi$ lacking fixed points. Theorem \ref{th:neumann-gevrey-secondorder} focuses on polynomials of second degree having a single fixed point and provides some restrictions on the indices $d, d'$ that are necessary for the resolvent operator to be continuous between the classes $\Sigma_d$ and $\Sigma_{d'}.$ Section \ref{sec:polynomials} contains some estimates involving iterates of a polynomial or derivatives of such iterates that play an essential role in the proof of the main theorems.
\par\medskip
We begin by describing the Gelfand-Shilov function spaces and some of their basic properties. We first give the definition of non-quasianalytic weight function in the sense of Braun, Meise and Taylor \cite{bmt}.
\begin{definition}\label{def:weight} A continuous increasing function $\omega :[0,\infty [\longrightarrow [0,\infty [$ is called a {\it weight} if it satisfies:
	\begin{itemize}
		\item[$(\alpha)$] there exists $K\geq 1$ with $\omega (2t) \leq K(\omega (t)+1)$ for
		all $t\geq 0$,
		\item[$(\beta)$] $\displaystyle\int_{0}^{\infty}\frac{\omega (t)}{1+t^{2}}\ dt < \infty $,
		\item[$(\gamma)$] $\log(1+t^{2})=o(\omega (t))$ as t tends to $\infty$,
		\item[$(\delta)$] $\varphi_\omega :t\rightarrow \omega (e^{t})$ is convex.
	\end{itemize}

\end{definition}
In the main results of the paper we consider weights satisfying a more restrictive condition than ($\alpha$), namely we consider weights that are subadditive ($\omega(a+b)\leq \omega(a) + \omega(b)$) or equivalent to a subadditive one. The main examples are $\omega(t) = t^{\frac{1}{d}}, d>1,$ or $\omega(t) = \max\left(0, \log^p t\right), p > 1.$
\par\medskip
The {\it Young conjugate} $\varphi_\omega ^{*}:[0,\infty [ \longrightarrow {\mathbb R}$  of
$\varphi_\omega$ is defined by
$$
\varphi_\omega^{*}(s):=\sup\{ st-\varphi_\omega (t):\ t\geq 0\},\ s\geq 0.
$$ Then $\varphi_\omega^\ast$ is convex, $\varphi_\omega^\ast(s)/s$ is increasing and $\displaystyle\lim_{s\to \infty}\frac{\varphi_\omega^\ast(s)}{s} = +\infty.$ The weight $\omega$ is said to be a {\it strong weight} if
\begin{itemize}
	\item[$(\varepsilon)$]
	there exists a constant $C \geq 1$ such that for all $y > 0$ the following inequality holds
	\begin{equation}
		\int_1^\infty \frac{\omega(yt)}{t^2}\ dt \leq C\omega(y) + C.
	\end{equation}
\end{itemize}

\begin{definition} Let $\omega$ be a weight function. The Gelfand-Shilov space of Beurling type $\mathcal{S}_{\omega}(\mathbb{R})$ consists of those functions $f\in C^\infty({\mathbb R})$ with the property that
	$$
	\|f\|_{\omega, \lambda}:= \sup_{x\in {\mathbb R}}\sup_{n,q\in {\mathbb N}_0}\left(1+|x|\right)^q|f^{(n)}(x)|\exp\left(-\lambda\varphi_\omega^\ast(\frac{n+q}{\lambda})\right) < \infty$$ for every $\lambda > 0.$
\end{definition}

$\mathcal{S}_{\omega}(\mathbb{R})$ is a nuclear Fréchet space (see \cite{nuclear}). Several equivalent systems of semi-norms describing the topology of $\mathcal{S}_{\omega}(\mathbb{R})$ can be found in \cite{asensio_jornet,seminormas}. In \cite{jmaa} one can find the connection with Gelfand-Shilov spaces defined in terms of a weight sequence instead of a weight function.
\par\medskip
$\omega(t) = t^{\frac{1}{d}}$ is a strong weight for each $d>1.$ The corresponding Gelfand-Shilov space is denoted $\Sigma_d({\mathbb R}) = {\mathcal S}_\omega({\mathbb R}).$
\par\medskip
Condition $(\varepsilon)$ is equivalent to the surjectivity of the Borel map
$$
B:{\mathcal S}_{(\omega)}({\mathbb R})\to {\mathcal E}_{(\omega)}(\{0\}), f\mapsto \left(f^{(j)}(0)\right)_{j\in {\mathbb N}_0},$$ where
$$
{\mathcal E}_{(\omega)}(\{0\}) = \left\{(x_j)_j\in {\mathbb C}^{{\mathbb N}_0}:  \ \sup_j|x_j|\exp(-k\varphi_\omega^\ast(\frac{j}{k})) < \infty\ \forall k > 0\right\}.
$$ The Borel map will be useful in the proof of Theorem \ref{th:neumann-gevrey-secondorder}.
	
\section{Some estimates on polynomials}\label{sec:polynomials}

The aim of this section is to obtain lower bounds for the iterations of a polynomial lacking fixed points (Proposition \ref{prop:iterativepoly}) and upper bounds for the derivatives of the iterates of such a polynomial (Proposition \ref{prop:upper_bound_derivatives}). First we need some auxiliary lemmas.

\begin{lemma}\label{lem:prod-1} For every $\bm k = (k_1,...,k_n)\in \mathbb{N}_0^n$ such that $\sum_{\ell=1}^n \ell k_\ell = n$ one has
	$$\prod_{\ell = 1}^{n}\ell^{\ell k_\ell} \leq \frac{n^n}{n!} \prod_{\ell=1}^n \ell!^{k_\ell}.$$
\end{lemma}
\begin{proof}
Without loss of generality we can assume $k_\ell \geq 1$ for all $1\leq \ell\leq n.$ According to the binomial theorem, for every $\alpha, \beta\in {\mathbb N}_0^q$ such that $|\alpha| = |\beta| = N$ we have
$$
N^N = \left(\alpha_1 + \ldots + \alpha_q\right)^N \geq \frac{N!}{\beta_1!\ldots \beta_q!}\alpha_1^{\beta_1}\ldots \alpha_q^{\beta_q}.$$ After choosing $\beta = \alpha$ we obtain
\begin{equation}\label{eq:binomio-1}
\prod_{p=1}^q \frac{\alpha_p^{\alpha_p}}{\alpha_p!} \leq \frac{N^N}{N!}.\end{equation} In particular, for $k_\ell \geq 1$ and $\alpha_p = \ell,\ 1\leq p\leq k_\ell,$
\begin{equation}\label{eq:binomio-2}
\left(\frac{\ell^\ell}{\ell!}\right)^{k_\ell} \leq \frac{(\ell k_\ell)^{\ell k_\ell}}{(\ell k_\ell)!}.\end{equation} Finally, after taking $\alpha_\ell = \ell k_\ell$ and $q = n$ in (\ref{eq:binomio-1}) we conclude
\begin{equation}\label{eq:binomio-3}
\prod_{\ell=1}^n \frac{(\ell k_\ell)^{\ell k_\ell}}{(\ell k_\ell)!} \leq \frac{n^n}{n!}.\end{equation} The conclusion follows from (\ref{eq:binomio-2}) and (\ref{eq:binomio-3}).
\end{proof}

\begin{lemma}\label{lem:combinatorics1} For every $n\in\mathbb{N}$ the following holds:
	\begin{equation}\label{suminversecombinatorics}
		\sum_{k=0}^{n} \binom{n}{k}^{-1} =\frac{n+1}{2^n}\sum_{k=0}^n \frac{2^k}{k+1} \leq 3.
	\end{equation}
\end{lemma}
\begin{proof}
	We denote $\displaystyle S_n:=\sum_{k=0}^{n} \binom{n}{k}^{-1}.$ From the relation between Gamma and Beta functions we have
		$$
		\begin{array}{*2{>{\displaystyle}l}}
			S_n & =(n+1)\sum_{k=0}^n B(k+1,n-k+1) = (n+1)\int_0^1 (1-t)^n \sum_{k=0}^n \left(\frac{t}{1-t}\right)^k dt \\ & \\ & =(n+1)\int_0^1 \frac{t^{n+1}-(1-t)^{n+1}}{2t-1}dt = \frac{n+1}{2^{n+1}}\int_{0}^1 \frac{(s+1)^{n+1}-(1-s)^{n+1}}{s}ds \\ & \\ & =\frac{n+1}{2^{n+1}} \left(\int_{0}^1 \sum_{k=0}^n (1+s)^k ds + \int_{0}^1 \sum_{k=0}^n (1-s)^k ds\right)\\ & \\ & = \frac{n+1}{2^{n+1}} \sum_{k=0}^n \left(\frac{2^{k+1}-1}{k+1}+\frac{1}{k+1}\right).
		\end{array}$$ Hence
		$$
		S_n  = \frac{n+1}{2^{n}} \sum_{k=0}^n \frac{2^{k}}{k+1}
		$$ for all $n\in\mathbb{N}$, as we wanted. Now,
		$$
		S_n = 1+\frac{n+1}{2n} S_{n-1}
		$$ for all $n\geq 2.$ By an easy computation we get $S_n\leq 3$ whenever $n\leq 6.$ Finally, assuming that $S_{n-1}\leq 3$ for some $n\geq 7$ we conclude
		$$
		S_n = 1+\frac{n+1}{2n} S_{n-1} < 1 + \left(\frac{1}{2} + \frac{1}{2n}\right)3 < 3.$$
	\end{proof}

Given two natural numbers $n$ and $k$ we consider the set
$$
H_{n,k}:=\left\{\bm k = (k_1,...,k_n)\in \mathbb{N}_0^n:\ \sum_{\ell = 1}^n k_\ell = k,\ \ \sum_{\ell=1}^n \ell k_\ell = n\right\}.$$ Observe that $H_{n,k} = \emptyset$ whenever $n < k.$ 
	
\begin{lemma}\label{lem:combinatorics2}
For every $n\geq k$ we have
\begin{equation}\label{eq:combinatoricsestimation}
\sum_{\bm k\in H_{n,k}}\ \prod_{\ell = 1}^n \ell!^{k_\ell} \leq n!\end{equation}
\end{lemma}	
\begin{proof} We first observe that for $n = k$ the set $H_{n,n}$ consists of the only element $\bm k = \left(n,0,\ldots,0\right).$ That is, $k_1 = n$ while $k_\ell = 0$ for every $\ell\geq 2.$  Therefore, in this case the inequality established by the lemma is obvious, so we will assume from now on that $n > k.$ We will proceed by induction on $k.$
\par\medskip
We start with the case $k=1$ by noticing that $H_{n,1}$ consists of the only element $\bm k=\left(0,...,0,1\right)\in {\mathbb N}_0^n$ for all $n > 1.$ Hence in this case
$$
\sum_{\bm k\in H_{n,k}}\ \prod_{\ell = 1}^n \ell!^{k_\ell} = n! $$
\par\medskip
We now fix $n > k\geq 2$ and assume that inequality (\ref{eq:combinatoricsestimation}) holds for $H_{q,k-1}$ whenever $q \geq k-1.$ We note that the first non-zero coordinate of $\bm k\in H_{n,k}$ cannot be the $n$th coordinate, since in that case we would have $\bm k = (0,\ldots,0,1)$ and therefore $k = 1,$ which contradicts the condition $k\geq 2.$ Hence we can decompose
$$
H_{n,k} = \bigcup_{j=1}^{n-1} I_j,$$ where $I_j$ denotes the subset of $H_{n.k}$ consisting of those $\bm k$ for which the first non-zero coordinate is precisely the $j$th coordinate. Obviously, the sets $I_1,...,I_{n-1}$ are pairwise disjoint. Some of the sets $I_j$ could be empty, as follows from the discussion below.
\par\medskip
{\it Claim:} Suppose $I_j\neq \emptyset.$ Then $k_\ell = 0$ for every $\ell > n-j$ whenever $\bm k = \left(k_1,\ldots,k_n\right)\in I_j.$
\par
We first consider the case that $n-j > j.$ For every $\ell > n-j$ we have 
$$
j(k_j-1) + \ell k_\ell = j k_j + \ell k_\ell -j \leq n-j,$$ which implies $k_\ell = 0.$ 
\par
If $n-j\leq j$ then, for every $\ell > j$ we have 
$$
 j(k_j-1) + \ell k_\ell = j k_j + \ell k_\ell -j \leq n-j \leq j,$$ which implies $k_\ell = 0.$ Hence $k_\ell\neq 0$ if, and only if, $\ell = j.$ Then $k_j = k\geq 2$ and $2j\leq jk_j\leq n\leq 2j.$ Hence $n=2j$ and $k_\ell = 0$ for $\ell > n-j = j.$ The Claim is proved.
 \par

 Observe that in particular we proved that $n$ is even whenever $I_j\neq \emptyset$ for some $j$ satisfying $n-j\leq j.$
  \par\medskip
  To each element $\bm k\in I_j$ ($1\leq j\leq n-1$) we associate $\tilde{\bm k} \in H_{n-j,k-1}$ obtained from $\bm k$ by removing all coordinates after the coordinate $n-j$ (which are null according to the Claim) and replacing $k_j$ (non-zero) with $k_j - 1.$ In fact, $$\displaystyle\sum_{\ell=1}^{n-j} \tilde{k}_\ell = \sum_{\ell=1}^n k_\ell -1 = k-1\ \  \mbox{while}\ \  \sum_{\ell=1}^{n-j}\ell \tilde{k}_\ell = \sum_{\ell=1}^n \ell k_\ell -j = n-j.$$ Observe that
  $$
  \prod_{\ell=1}^n \ell!^{k_\ell} = j! \prod_{\ell=1}^{n-j}\ell!^{\tilde{k}_\ell}.$$ Finally, from the induction hypothesis we obtain
  $$
  \begin{array}{*2{>{\displaystyle}l}}
 \sum_{\bm k\in H_{n,k}}\ \prod_{\ell = 1}^n \ell!^{k_\ell} & = \sum_{j=1}^{n-1}\sum_{\bm k\in I_j}\ \prod_{\ell = 1}^n \ell!^{k_\ell} \leq \sum_{j=1}^{n-1} j! \sum_{\tilde{\bm k}\in H_{n-j,k-1}}\ \prod_{\ell = 1}^{n-j} \ell!^{\tilde{k}_\ell}\\ & \\ & \leq \sum_{j=1}^{n-1} j! (n-j)! = n!\sum_{j=1}^{n-1}\binom{n}{j}^{-1} \leq n!	
\end{array}$$ The last inequality follows from Lemma \ref{lem:combinatorics1}. 	
\end{proof}

We recall that two polynomials $\psi, \phi$ are said to be linearly equivalent if there exists $\ell(x)=\alpha x+\beta$ for all $x\in\mathbb{R}$, $\alpha\not = 0,$ such that $\phi(x)=(\ell \circ \psi\circ\ell^{-1})(x)$ for all $x\in\mathbb{R}$. Every polynomial of even degree is linearly equivalent to a monic polynomial: if $\psi(x) = a x^{2p} + r(x)$ with $a\neq 0$ and $r(x)$ a polynomial of degree less than $2p$ then we choose $b$ satisfying $b^{2p-1} = a.$ Then, for $\ell(x) = bx,$ the polynomial $\ell \circ \psi\circ\ell^{-1}$ is monic.
\par\medskip
Below we obtain lower bounds for the iterations of a polynomial lacking fixed points.

\begin{proposition}\label{prop:iterativepoly}
	Let $\psi$ be a polynomial of degree $2p$ with $p\geq 1$, without fixed points. Then, for each $b>1$ there is $m_0\in\mathbb{N}$ such that \begin{equation*}
		|\psi_{m_0+k}(x)|\geq b^{2^k}
	\end{equation*} for all $x\in\mathbb{R}$ and for all $k\in\mathbb{N}$.
\end{proposition}
\begin{proof}
	In the case of $p=1$, $\psi$ is linearly equivalent to $\phi(x)=x^2+c$ with $c>\frac{1}{4}$ (see \cite[Section 3]{advances}), while $\psi$ is linearly equivalent to $\phi(x)=x^{2p}+r(x)$ with $r$ being a polynomial of degree less than $2p$ for $p\geq 2$. Indeed, we write $\psi(x)=a_{2p}x^{2p}+a_{2p-1}x^{2p-1}+...+a_1 x+a_0$ with $a_{2p}\not =0$ and observe that if $\ell^{-1}(x)=  a_{2p}^{\frac{1}{2p-1}} x$ we have that \begin{equation*}
		(\ell^{-1}\circ \psi \circ \ell)(x)=a_{2p}^{\frac{1}{2p-1}} \psi(\frac{x}{a_{2p}^{\frac{1}{2p-1}}})=x^{2p}+r(x)
	\end{equation*} with $r$ being a polynomial of degree less than $2p$, as it was required. In both cases, there is $B>1$ (to be chosen later) so that $\phi(x)\geq x^2$ for all $|x|\geq B$. Hence, $\phi_m(x)\geq x^{2^m}\geq B^{2^m}$ for all $|x|\geq B$. Since $\phi(x)>x$ for all $x\in\mathbb{R}$ and $\lim_{|x|\to \infty} (\phi(x)-x)=+\infty,$ there is $a>0$ such that $\phi(x)>x+a$ for all $x\in\mathbb{R}$ and therefore, $\phi_m(x)>x+ma$ for all $x\in\mathbb{R}$ and $m\in\mathbb{N}$. In particular, $\phi_m(x)>\min\{\phi(x): x\in\mathbb{R}\}+(m-1)a$ for all $m\geq 2$ and $x\in\mathbb{R}$. Choose $m_0\geq 2$ so that $\phi_{m_0}(x)\geq B$ for all $x\in\mathbb{R}$. By induction on $k$ we deduce that \begin{equation*}
		\phi_{m_0+k}(x)\geq B^{2^k}
	\end{equation*} for all $x\in\mathbb{R}$, $k\in\mathbb{N}$. Now, take $\ell(x)=ex+d$ so that $\psi=\ell\circ \phi\circ \ell^{-1}$. Fix $b>0$. After choosing $B>\max\{b, \frac{|d|+1}{|e|}b\}$ we get that \begin{equation*}
		(\frac{B}{b})^{2^k}\hspace{0.05cm}|e|-1>|d|
	\end{equation*} for all $k\in\mathbb{N}_0$. Then, \begin{equation*}
		|\psi_{m_0+k}(x)|=|\ell(\phi_{m_0+k}(\ell^{-1}(x)))|\geq |e|B^{2^k}-|d|\geq b^{2^k}
	\end{equation*} for all $x\in\mathbb{R},$ $k\in\mathbb{N}$.
\end{proof}
\par\medskip It follows from the previous proof that every polynomial without fixed points is linearly equivalent to a monic polynomial without fixed points.
Now we obtain an upper bound for the derivatives of the successive iterations of a polynomial without fixed points.

\begin{proposition}\label{prop:upper_bound_derivatives}
Let $\psi$ be a polynomial of degree greater than $1$ without fixed points. For every $\alpha > 1$ there exist $C > 0$ and $r > 1$ such that
$$
|\psi^{(n)}_m(x)|\leq C r^n n!^2 \left(1 + |\psi_m(x)|\right)^\alpha$$ for all $x\in \mathbb{R}$, $n\in\mathbb{N},m\in\mathbb{N}$.
\end{proposition}
\begin{proof}
Without loss of generality we can assume that $\psi$ is a monic polynomial of even degree $2p.$ Choose $0 < c < \left(\frac{1}{(2p)!}\right)^{\frac{1}{2p-1}}.$ There exists $x_0 > 0$ such that
\begin{equation}\label{eq:upper-bound}
\sum_{k=1}^{2p} c^{k-1}\left(1+|x|\right)^{\alpha k}\cdot|\psi^{(k)}(x)| \leq \left(1 + \left|\psi(x)\right|\right)^\alpha\ \ \forall\ x\geq x_0.\end{equation} In fact, since $\alpha > 1,$ the term in the left hand side behaves like $\lambda\left(1 + |x|\right)^{2p\alpha}$ as $|x|\to \infty,$ where $\lambda = (2p)! c^{2p-1} < 1,$ while the term in the right hand side behaves like $\left(1 + |x|\right)^{2p\alpha}.$ From Proposition \ref{prop:iterativepoly} there is $m_0\in {\mathbb N}$ such that
$$
\psi_m(x)\geq x_0\ \ \forall x\in {\mathbb R},\ m\geq m_0.$$ Take $D > 1$ such that
$$
|\psi^{(n)}_m(x)|\leq D \left(1 + |\psi_m(x)|\right)^\alpha\ \ \forall\ 1\leq m\leq m_0,\ x\in {\mathbb R}, \ \in\mathbb{N}.$$ Finally we put $r:=\frac{D}{c} > 1$ and prove by induction that
\begin{equation}\label{eq:induction_upper-bound}
|\psi_m^{(n)}(x)|\leq c \hspace{0.02cm} n!\hspace{0.02cm} r^n \hspace{0.02cm} n^n \hspace{0.02cm}\left(1 + |\psi_m(x)|\right)^\alpha
\end{equation} for all $x\in {\mathbb R},$ $n\in {\mathbb N},\ m\in {\mathbb N}.$ Due to the choice of constants, the inequality (\ref{eq:induction_upper-bound}) is satisfied for $1\leq m\leq m_0,\ n\in {\mathbb N}$ and $x\in {\mathbb R}.$ Let us now assume that (\ref{eq:induction_upper-bound}) is satisfied for some $m\geq m_0$ and every $n\in {\mathbb N}.$ Then, for every $n\geq k$ and $\bm k\in H_{n,k}$ we have
$$
\begin{array}{*2{>{\displaystyle}l}}
\prod_{\ell=1}^n  |\frac{\psi_m^{(\ell)}(x)}{\ell!}|^{k_\ell} & \leq c^k r^n \left(1 + |\psi_m(x)|\right)^{\alpha k} \prod_{\ell = 1}^n \ell^{\ell k_\ell}	\\ & \\ & \underset{\textsf{Lemma \ref{lem:prod-1}}}{\leq} c^k r^n \left(1 + |\psi_m(x)|\right)^{\alpha k} \frac{n^n}{n!}\prod_{\ell=1}^n \ell!^{k_\ell}.
\end{array}$$ Consequently
$$
\begin{array}{*2{>{\displaystyle}l}}
|\psi_{m+1}^{(n)}(x)| & = |\left(\psi\circ\psi_m\right)^{(n)}(x)| \leq \sum_{k=1}^{2p}\sum_{\bm k\in H_{n,k}}\frac{n!}{k_1!\ldots k_n!}|\psi^{(k)}(\psi_m(x))|\cdot \prod_{\ell=1}^n  |\frac{\psi_m^{(\ell)}(x)}{\ell!}|^{k_\ell}\\ & \\ & \leq c r^n n^n \sum_{k=1}^{2p} c^{k-1}\left(1 + |\psi_m(x)|\right)^{\alpha k}|\psi^{(k)}(\psi_m(x))|\sum_{\bm k\in H_{n,k}}\prod_{\ell=1}^n \ell!^{k_\ell}\\\ & \\ & \underset{\textsf{Lemma \ref{lem:combinatorics2}}}{\leq} c r^n n^n n! \sum_{k=1}^{2p} c^{k-1}\left(1 + |\psi_m(x)|\right)^{\alpha k}|\psi^{(k)}(\psi_m(x))|
\end{array}$$ Since $\psi_m(x)\geq x_0$ we can apply inequality (\ref{eq:upper-bound}) to finally conclude that
$$
|\psi_{m+1}^{(n)}(x)|\leq c r^n n^n n!\left(1 + |\psi_{m+1}(x)|\right)^{\alpha}$$ and the induction process is complete. The conclusion now follows from Stirling's formula.
\end{proof}	

\section{Power bounded composition operators}\label{sec:power}
We recall that a continuous linear operator $T:X\to X$ on a Fréchet space $X$ is said to be power bounded if the sequence of iterates $\left(T^m\right)_m$ is equicontinuous. A related concept to power boundedness is that of mean ergodicity. The operator $T:X\to X$ is mean ergodic if the limit
$$
Px:=\lim_n\frac{1}{n}\sum_{k=1}^n T^k(x)$$ exists for every $x\in X.$ Every power bounded operator $T$ on a reflexive Fréchet space $X$ is mean ergodic (see \cite{ABR}), which implies $\displaystyle\lim_n \frac{T^n(x)}{n} = 0$ for every $x\in X.$ There are many articles dedicated to the study of the behavior of the sequence of iterations of a continuous operator in a function space, and especially when said operator is a composition operator.
\par\medskip
The problem we address here is different. It is obvious that if $\psi$ is a non-constant polynomial then the composition operator $C_\psi:f\mapsto f\circ \psi$ maps the Schwartz class ${\mathcal S}({\mathbb R})$ into itself. However, as proven in \cite[Corollary 3.8]{jmaa}, $C_\psi$ does not leave invariant the classical Gelfand-Shilov spaces $\Sigma_d({\mathbb R})$ unless $\psi$ has degree one. On the other hand, if $\omega$ is a subadditive weight and $\sigma(t) = \omega(t^{\frac{1}{2}})$ then $C_\psi\left({\mathcal S}_\omega({\mathbb R})\right) \subset {\mathcal S}_\sigma({\mathbb R})$ (\cite[Theorem 4.4 ]{jmaa}). In particular $C_\psi\left(\Sigma_d({\mathbb R})\right)\subset \Sigma_{2d}({\mathbb R}).$
\par\medskip
Since each iterate $C_{\psi}^m$ can be written as $C_{\psi}^m = C_{\psi_m},$ where
$\psi_m=\psi \circ \overset{(m)}{\dots}\circ \psi$ is a polynomial, we conclude that $\left(C_{\psi}^m\right)_m$ is a sequence of continuous operators from ${\mathcal S}_\omega({\mathbb R})$ into ${\mathcal S}_\sigma({\mathbb R})$. So it is natural to ask whether this family of operators is equicontinuous or not. Since the sequence of operators acts between two different Fréchet spaces, most of the known results are not useful for our purposes, in particular the connection with mean ergodicity.
\par\medskip
Let us begin considering the case that $\psi$ is a polynomial of degree $1.$ Then $C_\psi\left({\mathcal S}_\omega({\mathbb R})\right) \subset {\mathcal S}_\omega({\mathbb R}).$

\begin{proposition} Let $\psi(x) = ax + b,$ with $a\neq 0.$ The following are equivalent:
\begin{enumerate}
\item $C_\psi:{\mathcal S}_\omega({\mathbb R})\to {\mathcal S}_\omega({\mathbb R})$ is power bounded.
\item $C_\psi:{\mathcal S}_\omega({\mathbb R})\to {\mathcal S}_\omega({\mathbb R})$ is mean ergodic.
\item $\left(C_{\psi_m}\right)_m$ is equicontinuous in $\mathcal{L}(\mathcal{S}_{\omega}(\mathbb{R}), \mathcal{S}(\mathbb{R})).$
\item $\psi(x) = x$ or $\psi(x) = -x + b.$
\end{enumerate}	
\end{proposition}
\begin{proof}
$(1)\Rightarrow (3)$ is obvious. $(1)\Rightarrow (2)$ is clear because ${\mathcal S}_\omega({\mathbb R})$ is reflexive and $(4)\Rightarrow (1)$ holds since $\{\psi_{m}: m\in\mathbb{N}\}$ is a finite set of affine functions. We now assume that $(4)$ is not satisfied. Take $f\in {\mathcal S}_\omega({\mathbb R})$ such that $f'(-1)=f'(1)=f'(0)=1.$ We will check that $\left(\frac{1}{m}C_{\psi_m}f\right)_m$ is unbounded in $\mathcal{S}(\mathbb{R}),$ which implies that neither (2) nor (3) are satisfied. To this end we distinguish several cases.
 \par If $a=1$ then $\psi_m(x) = x+mb$ with $b\neq 0.$ For $x_m:=-bm$ we have
 $$
 \lim_m \frac{x_m^2}{m}\left|\left(C_{\psi_m}f\right)'(x_m)\right| = \infty$$ and we are done.
 \par
 If $a\neq 1$ then $\psi(x)$ is linearly equivalent to $\phi(x) = ax.$ We take $x_m = |a|^{-m}.$ In the case $0<|a|<1$ we have
 $$
 \lim_{m}\frac{x_m^2}{m}\left|\left(C_{\phi_m}f\right)'(x_m)\right| = \lim_m \frac{x_m}{m} = \infty,$$ while for $|a| > 1$ we have
 $$
 \lim_m \frac{1}{m}\left|\left(C_{\phi_m}f\right)'(x_m)\right| = \lim_m \frac{|a|^m}{m} = \infty.$$ The proof is complete.
\end{proof}	
We will need the following result: \begin{proposition}
    \label{conditionpowerboundedness} Let $\omega$ be a subadditive weight, $a\geq 1$, $\sigma(t)=\omega(t^{\frac{1}{a}})$ for all $t\geq 0$. Assume that \begin{enumerate}
        \item there is $C_0>0$ such that $|x|\leq C_0 (1+|\psi_m(x)|)^a$ for all $x\in\mathbb{R}$, $m\in\mathbb{N}$,
        \item for all $\lambda>0$ there is $C_\lambda>0$ such that \begin{equation*}
            |\psi^{(n)}_m(x)|\leq C_\lambda \hspace{0.1cm} \exp(\lambda \varphi_\sigma^*(\frac{n}{\lambda}))\hspace{0.1cm}(1+|\psi_m(x)|)^p
        \end{equation*} for all $j\in\mathbb{N}$, $x\in\mathbb{R}$, $m\in\mathbb{N}$, where $p=a-1$. 
    \end{enumerate} Then, $C_\psi: \mathcal{S}_{\omega}(\mathbb{R}) \to \mathcal{S}_{\sigma}(\mathbb{R})$ is power bounded. 
\end{proposition}
\begin{proof} We can adapt the proof of Proposition 4.1 in \cite{jmaa} to prove that $\left(C_{\psi_m}f\right)_m$ is a bounded sequence in ${\mathcal S}_\sigma({\mathbb R})$ for each $ f\in {\mathcal S}_\omega({\mathbb R}).$ The Banach-Steinhaus theorem gives the conclusion.
\end{proof}

Concerning polynomials of degree greater than one we have the following characterization. To prove $(2)\Rightarrow(3)$ we essentially follow the idea of \cite[Theorem 3.11]{fgj}. We present the details of the proof because the argument of \cite{fgj} is based on the fact that every power bounded operator $T:X\to X$ on a reflexive Fréchet space is mean ergodic. The proof of $(3)\Rightarrow (1)$ is very different from that of \cite{fgj} and is based on Proposition \ref{prop:upper_bound_derivatives}.

\begin{theorem}\label{teo:power-bounded} Let $\omega$ be any subadditive weight and $\sigma(t)= \omega(t^{\frac{1}{a}})$ for $a>2.$ Given a polynomial $\psi$ of degree greater than one, the following statements are equivalent:\begin{enumerate}
		\item $\left(C_{\psi_m}\right)_m$ is equicontinuous in $\mathcal{L}(\mathcal{S}_{\omega}(\mathbb{R}), \mathcal{S}_\sigma(\mathbb{R})).$
		\item $\left(C_{\psi_m}\right)_m$ is equicontinuous in $\mathcal{L}(\mathcal{S}_\omega(\mathbb{R}), \mathcal{S}(\mathbb{R})).$
		\item $\psi$ lacks fixed points.
		\end{enumerate}
\end{theorem}
\begin{proof} According to \cite[Theorem 4.4]{jmaa}, $C_{\psi_m}\left(\mathcal{S}_{\omega}(\mathbb{R})\right) \subset \mathcal{S}_{\sigma}(\mathbb{R})$ for every $m\in {\mathbb N}.$
\par
$(1)\Rightarrow (2)$ is obvious.
\par
$(2)\Rightarrow (3).$ Let us assume that $\psi$ has some fixed point. There is $K > 0$ such that $|\psi(x)|\geq 2|x|$ for every $|x|\geq K,$ hence $|\psi_m(x)|\geq 2^m |x|$ for all $m\in {\mathbb N}$ and $|x|\geq K.$ In particular $\displaystyle\lim_m \psi_m(x) = \infty$ whenever $|x|\geq K.$ We consider
$$
A=\left\{x\in {\mathbb R}:\ \left(\psi_m(x)\right)_m\ \mbox{does not diverge to}\ \infty\right\}.$$ Since $\psi$ has some fixed point then the set $A$ is non-empty and, due to the previous considerations, it is bounded. Take $b = \sup\{|x|:\ x\in A\}$ and let $f\in \mathcal{S}_{\omega}(\mathbb{R})$ be given such that $f(x) = 1$ for all $|x|\leq b.$
By condition $(2),$ the sequence $\left(C_{\psi_m}f\right)_m$ is bounded in the Montel space $\mathcal{S}(\mathbb{R}),$ hence it admits a subsequence that converges pointwise to some $g\in \mathcal{S}(\mathbb{R}).$ For simplicity we denote the subsequence as the entire sequence. There is a sequence $\left(x_n\right)_n\subset A$ that converges to either $b$ or $-b.$ Moreover, $\psi_m(x_n)\in A\subset [-b,b]$ for every $m,n\in {\mathbb N},$ hence
$$
g(x_n) = \lim_m f\left(\psi_m(x_n)\right) = 1\ \ \forall n\in {\mathbb N}.$$ If $\left(x_n\right)_n$ converges to $b$ then $g(b) = 1.$ However, for $|x| > b$ the sequence $\left(\psi_m(x)\right)_m$ diverges to $\infty,$ which implies
$$
g(x) = \lim_m f\left(\psi_m(x)\right) = 0.$$ This is a contradiction. In the case that $\left(x_n\right)_n$ converges to $-b$ we argue similarly.
\par
$(3)\Rightarrow (1).$ We only need to check that the two conditions appearing in Proposition \ref{conditionpowerboundedness} hold:\\

(i) Since $\psi$ has even degree it is equivalent to a monic polynomial lacking fixed points. Hence, we will assume that $\psi$ itself is monic.
As in the proof of $(2)\Rightarrow (3),$ there is $K > 1$ such that $|\psi_m(x)|\geq 2^m |x|$ for all $m\in {\mathbb N}$ and $|x|\geq K.$ Then
$$
|x|\leq 1 + |\psi_m(x)|\ \ \forall\ |x|\geq K,\ m\in {\mathbb N}$$ and condition (i) is satisfied with $C_0 = K.$
\par
(ii) We recall that $\varphi_\sigma^\ast(s) = \varphi_\omega^\ast(as),\ s\geq 0.$ Take $\alpha = p > 1$ in Proposition \ref{prop:upper_bound_derivatives}. There exist $C > 0$ and $r > 1$ such that
$$
|\psi^{(n)}_m(x)|\leq C r^n n!^2 \left(1 + |\psi_m(x)|\right)^p$$ for all $x\in \mathbb{R}$, $n\in\mathbb{N},m\in\mathbb{N}.$ By \cite[Lemma A.1.(ii), (viii)]{paley}, for every $\lambda > 0$ there exist $B_\lambda > 0$ such that
$$
r^n n!^2 \leq B_\lambda\exp\left(\lambda\varphi_\omega^\ast\left(\frac{2n}{\lambda}\right)\right) \leq B_\lambda\exp\left(\lambda\varphi_\sigma^\ast\left(\frac{n}{\lambda}\right)\right).$$ Hence, condition (ii) is satisfied with $C_\lambda = C B_\lambda$ and the proof is complete.
\end{proof}

\begin{proposition}\label{prop:mean-ergodic}
Let $\omega$ be any subadditive weight and $\sigma(t)= \omega(t^{\frac{1}{a}})$ for $a>2.$ Given a polynomial $\psi$ of degree greater than one, the following statements are equivalent:
\begin{enumerate}
\item $\psi$ lacks fixed points.
\item $\displaystyle\lim_m f\circ \psi_m = 0$ in ${\mathcal S}_\sigma({\mathbb R})$ for every $f\in {\mathcal S}_\omega({\mathbb R}).$
\item $\displaystyle\lim_n \frac{1}{n}\sum_{m=1}^n f\circ \psi_m$ exists in ${\mathcal S}_\sigma({\mathbb R})$ for every $f\in {\mathcal S}_\omega({\mathbb R}).$
\item $\displaystyle\lim_n \frac{1}{n}\sum_{m=1}^n f\circ \psi_m$ exists in ${\mathcal S}({\mathbb R})$ for every $f\in {\mathcal S}_\omega({\mathbb R}).$
\end{enumerate}
\end{proposition}
\begin{proof}
$(1)\Rightarrow (2).$ Since $\psi$ lacks fixed points then either $\psi(x) > x + b$ for every $x\in {\mathbb R}$ and some $b > 0$ or $\psi(x) < x - b$ for every $x\in {\mathbb R}$ and some $b > 0.$ In either case we have $\displaystyle\lim_m \textcolor{red}{|}\psi_m(x)\textcolor{red}{|}= \infty$ for all $x\in {\mathbb R},$ which implies $\displaystyle\lim_m f\left(\psi_m(x)\right) = 0$ for every $x\in {\mathbb R}.$ Now the conclusion follows from the fact that $\left(f\circ\psi_m\right)_m$ is a bounded sequence in the Fréchet Montel space ${\mathcal S}_\sigma({\mathbb R})$ (Theorem \ref{teo:power-bounded}) and its only possible accumulation point is the zero function.
\par
$(2)\Rightarrow (3)\Rightarrow (4)$ are obvious. To show $(4)\Rightarrow (1)$ we can proceed as in the proofs of \cite[Proposition 3.4 and Theorem 3.11 $(2)\Rightarrow(3)$]{fgj} with the obvious changes.
\end{proof}
We do not know whether the above results are also true for $a=2$. 

It may be worth making the above results explicit in the case where the weight $\omega$ is a power of the logarithm. In this case, keeping the notation of the previous result,  ${\mathcal S}_\omega({\mathbb R})={\mathcal S}_\sigma({\mathbb R}).$ The limit case $p = 1$ corresponds to \cite[Theorem 3.11]{fgj}, since in this case ${\mathcal S}_\omega({\mathbb R})={\mathcal S}({\mathbb R})$, despite of the fact that $\omega$ would not be strictly speaking a weight function (Definition \ref{def:weight}$(\gamma)$ does not hold). 

\begin{corollary}\label{cor:powers-logarithm}
	Let $\omega(x)=\max\{0,\log ^p(x)\}$ with $p>1.$ Given a polynomial $\psi$ of degree greater than one, the following statements are equivalent:\begin{enumerate}
\item $C_\psi:{\mathcal S}_\omega({\mathbb R})\to {\mathcal S}_\omega({\mathbb R})$ is power bonded.
\item $\psi$ lacks fixed points.
\item $C_\psi:{\mathcal S}_\omega({\mathbb R})\to {\mathcal S}_\omega({\mathbb R})$ is mean ergodic.
\end{enumerate}
\end{corollary} 		

\section{Spectrum and Neumann series}\label{sec:spectrum}

For the spectral theory of linear operators not necessarily everywhere defined on non-normable spaces we follow \cite{ABR}. We will always work with Fr\'echet spaces, hence we will give the definitions in this setting. We will write $\lambda-T$ as a shorthand for $\lambda I - T,$ where $I$ denotes the identity operator.

\begin{definition}
Let $E$ be a Fréchet space and $T:D(T)\subset E \to E$ a linear operator. The resolvent set of $T$ is
$$
\rho(T) = \left\{\lambda\in {\mathbb C}:\ \lambda-T:D(T)\to E\ \mbox{is bijective and}\ (\lambda-T)^{-1}\in {\mathcal L}(E)\right\},$$ and the spectrum of $T$ is defined by $\sigma(T)={\mathbb C}\setminus \rho(T).$

$\rho^\ast(T)$ consists of those $\lambda\in \rho(T)$ for which there exists $V(\lambda)$ open neighbourhood of $\lambda,$ $V(\lambda)\subset \rho(T),$ such that $\left\{(\mu-T)^{-1}:\mu\in V(\lambda)\right\}$ is equicontinuous in ${\mathcal L}(E).$

We put $\sigma^\ast(T) = {\mathbb C}\setminus \rho^\ast(T),$ and we call it the Waelbroeck spectrum by similarity with the case in which $T\in {\mathcal L}(T).$

\end{definition}
\par\medskip
According to \cite[Remark 3.1]{ABR}, if $\rho(T) \neq \emptyset$ then $T$ is a closed operator, that is, if $(x_n)_n\subset D(T)$ converges to $x$ in $E$ and $(Tx_n)_n$ converges to $y,$ then $x\in D(T)$ and $y=Tx.$  Then, when  $\rho^\ast(T)\neq \emptyset ,$ \cite[Proposition 3.4]{ABR} can be applied and    the map $\rho^\ast(T)\to {\mathcal L}_b(E), \lambda\mapsto (\lambda-T)^{-1},$ is holomorphic.

\begin{proposition}\label{Walbroeck} Let $E$ and $F$ be Fr\'echet spaces with $E$ continuously included in $F.$ Given  $T\in L(F),$ we put  $D(A):=\{x\in E: Tx \in E\},$ and $A=T_{|_{D(A)}}:D(A)\to E.$  If $\sigma^\ast(A)$ is compact and there is  $R>0$   such that the series $\displaystyle \sum_{m=0}^\infty\frac{T^mx}{\mu^{m+1}}$ converges in $F$ for every $x\in F$ and $|\mu|>R,$ then $D(A)=E.$

\end{proposition}

\begin{proof} The hypothesis imply that $(\rho ^m T^mx)_m$ is bounded in $F$ for every $x\in F$ and $|\rho |<\frac{1}{R}.$ Hence, $\displaystyle \sum_{m=0}^\infty\frac{T^mx}{\mu^{m+1}}$ is absolutely convergent  in $F$ for every $x\in F$ and $|\mu |>R.$ From the compactness of $\sigma^\ast(A)$  we can find $r > R$ such that $\left\{\mu\in {\mathbb C}:\ |\mu| = r\right\}\subset \rho^\ast(A).$ Then the map
$$
\{|\mu |=r \}\to L(E), \, \mu \mapsto (\mu -A)^{-1}$$ is continuous and, for each $y\in E,$ the integral $\displaystyle\frac{1}{2\pi i}\int_{|\mu |=r}\mu (\mu-A)^{-1}y\ d\mu$ defines an element of $E.$ On the other hand, with convergence in $F,$ $$\mu (\mu-A)^{-1}y=\sum_{m=0}^\infty\frac{T^my}{\mu^{m}},$$ therefore $$\frac{1}{2\pi i}\int_{|\mu |=r}\mu (\mu-A)^{-1}y\ d\mu=\sum_{m=0}^\infty \frac{T^my}{2\pi i}\int_{|\mu |=r}\frac{1}{\mu^{m}}\  d\mu=Ty.$$
This finishes the proof.

\end{proof}

By \cite{advances}, given a polynomial $\psi$ of degree greater than one and without fixed points, $\sigma_{{\mathcal S}({\mathbb R})}(C_\psi)=\{0\},$ that is, for $\mu \neq 0,$ $(\mu -C_\psi )$ is a topological isomorphism onto ${\mathcal S}({\mathbb R}).$ Besides, $$(\mu -C_\psi )^{-1}f=\sum_{m=0}^\infty \frac{C_{\psi_m}f}{\mu^{m+1}},$$ with convergence in ${\mathcal S}({\mathbb R}).$ Using the estimates in Proposition \ref{prop:iterativepoly}, i\textcolor{red}{t} is easy to check that $\rho(C_\psi)=\rho^\ast(C_\psi) = {\mathbb C}\setminus\{0\}$ and, consequently, its Waelbroeck spectrum reduces to $\{0\}.$  Now, our next result follows immediately from Proposition \ref{Walbroeck} and means that, concerning the Waelbroeck spectrum, the behaviour of $C_\psi$ in Gelfand-Shilov  classes or in the Schwartz class  might be very different.

\begin{theorem} Let $\psi$ be a polynomial of degree greater than one and without fixed points. For $d>2,$ take  $E=\Sigma_d({\mathbb R}),$ $F={\mathcal S}({\mathbb R})$  and $T=C_\psi.$  With the notation of Proposition \ref{Walbroeck}, $\sigma^\ast(A)$ is not compact.

\end{theorem}

We observe that $D(A)$ is a dense subspace of $\Sigma_d({\mathbb R})$ as it contains $\Sigma_{d/2}({\mathbb R}).$
\par\medskip 

 Let us write $R_\mu=(\mu -C_\psi )^{-1}.$  Given a  subadditive weight $\omega,$ by Theorem \ref{teo:power-bounded}, $(C_{\psi_m})_m $ is equicontinuous in ${\mathcal L}( {\mathcal S_\omega}({\mathbb R}) ,  {\mathcal S_\sigma}({\mathbb R})),$ where $\sigma(t)= \omega(t^{\frac{1}{a}})$ for $a>2.$ Hence, it is natural to investigate whether $R_\mu \in  {\mathcal L}( {\mathcal S_\omega}({\mathbb R}) ,  {\mathcal S_\sigma}({\mathbb R})),$ that is, if the regularity of the solution $g$ of the equation $\mu g-C_\psi g = f$ depends on the regularity of the datum $f.$  This is the content of the next result.

\begin{proposition}\label{prop:resolvente-continua} Let $\omega$ be any subadditive weight and $\sigma(t)= \omega(t^{\frac{1}{a}})$ for $a>2.$ Given a polynomial $\psi$ of degree greater than one and without fixed points and $\mu \neq 0,$ we have that $R_\mu \in  {\mathcal L}( {\mathcal S_\omega}({\mathbb R}) ,  {\mathcal S_\sigma}({\mathbb R})).$
\end{proposition}

\begin{proof} Given $f \in {\mathcal S}_{\omega}({\mathbb R})$ and $\mu \neq 0,$ and $\lambda >0$ we will show that the series $\displaystyle \sum_{m=0
}^\infty \frac{\|C_{\psi_m}f\|_{\sigma,\lambda}}{\mu^m}$ converges.  Then, the result will follow by Banach-Steinhaus theorem and the fact that $\mu R_\mu f = \sum_{m=0
}^\infty \frac{C_{\psi_m}f}{\mu^m}.$

Let $b>1.$ By Proposition \ref{prop:iterativepoly}, there is $m_0\in\mathbb{N}$ such that \begin{equation*}
		|\psi_{m_0+k}(x)|\geq b^{2^k}
	\end{equation*} for all $x\in\mathbb{R}$ and for all $k\in\mathbb{N}$.
	
	Proceeding as in the proof of (3)$\Rightarrow $(1) of Theorem \ref{teo:power-bounded} to check that condition (ii) in Proposition \ref{conditionpowerboundedness} holds, we have that for all $\rho>0$ there is $C_\rho>0$ such that \begin{equation*}
	|\psi^{(n)}_m(x)|\leq C_\rho \hspace{0.1cm} \exp(\rho \varphi_\sigma^*(\frac{n}{\rho}))\hspace{0.1cm}(1+|\psi_m(x)|)^p
\end{equation*} for all $n\in\mathbb{N}$, $x\in\mathbb{R}$, $m\in\mathbb{N}$, where $p=a-1$. 
Proceeding as in the proof of Proposition 4.1 in \cite{jmaa} we obtain that given $\rho>0$ there exist positive numbers  $\rho'$ and $C$ such that $$(1+|\psi_m(x)|)^q|(f\circ \psi_m)^{(n)}(x)|\leq C \|f\|_{\omega,\rho'}\exp\left(\rho\varphi^\ast_\sigma(\frac{n+q}{\rho})\right),$$
	for all $x\in {\mathbb R},$ $n\in {\mathbb N},$ $\in{\mathbb N}$, $q\in {\mathbb N}_0$. Moreover, there is $C_0>0$ such that $1+|x|\leq C_0 (1+|\psi_m(x)|)$ for all $x\in\mathbb{R}$, $m\in\mathbb{N}$.
Hence, if $m=m_0+k,$
$$
\begin{array}{*2{>{\displaystyle}l}} b^{2^k}(1+|x|)^q |(f\circ \psi_m)^{(n)}(x)| & \leq  C_0^q (1+|\psi_m(x)|)^{q+1}|(f\circ \psi_m)^{(n)}(x)| \\ & \\ & \leq C_0^q C \|f\|_{\omega,\rho'}\exp\left(\rho\varphi^\ast_\sigma(\frac{n+q+1}{\rho})\right).\end{array}$$

Given $\lambda >0$ we choose $\rho>0$ and $C'>0$ such that $$ C_0^q C  \exp\left(\rho\varphi^\ast_\sigma(\frac{n+q+1}{\rho})\right) \leq C' \exp\left(\lambda \varphi^\ast_\sigma(\frac{n+q}{\lambda})\right), \forall x\in {\mathbb R},\,  \forall n, q\in {\mathbb N}_0.$$ This is possible by the convexity of $\varphi^\ast_\sigma$ and Lemma 3.3 \cite{jmaa}.
Then, for $m=m_0+k,$ $$\|C_{\psi_m}
(f)\|_{\sigma,\lambda}\leq \frac{1}{b^{2^{m-m_0}}} \hspace{0.05cm}C' \|f\|_{\omega, \rho'},$$  from where the convergence of $\displaystyle \sum_{m=0
}^\infty \frac{\|C_{\psi_m}f\|_{\sigma,\lambda}}{\mu^m}$ follows.

\end{proof}

For polynomials of degree greater than one and having fixed points, arguing as in \cite[Theorems 2.8,  2.10]{advances} we have
\begin{proposition} Let $\omega$ be a subadditive weight and $\psi$  a polynomial of degree greater than one and having fixed points. Then, for every $0<|\mu|\leq 1,$ there is $f \in {\mathcal S}_\omega({\mathbb R})$ such that no $g\in {\mathcal S}({\mathbb R})$ satisfies $\mu g - C_\psi g=f.$  If in addition $\psi$ has a fixed point $a$ such that $\psi'(a)>1$ and $\psi^{(n)}(a)\geq 0$ for all $n\geq 2,$ then, for every $\mu \neq 0$ there is $f\in {\mathcal S}_\omega({\mathbb R})$ such that the equation $\mu g - C_\psi g=f$ has no solution in ${\mathcal S}({\mathbb R}).$

\end{proposition}

Now we focus on polynomials $\psi$ of degree 2. In this case, the polynomial is linearly equivalent to $x^2+c.$ The parameter $c$ depends on the number of fixed points of the polynomial. In fact, $\psi$ has two different fixed points if and only if $c<\frac{1}{4},$ has no fixed points when $c>\frac{1}{4}$ and has a unique fixed point for $c=\frac{1}{4}.$ According to this, and by our previous results, when $c>\frac{1}{4},$ $R_\mu \in  {\mathcal L}( {\mathcal S_\omega}({\mathbb R}) ,  {\mathcal S_\sigma}({\mathbb R})),$ for every subadditive  weight and $\sigma(t)= \omega(t^{\frac{1}{a}})$ with $a>2,$ whereas for  $c<\frac{1}{4},$ we may find $f\in {\mathcal S}_\omega({\mathbb R})$ such that the equation $\mu g - C_\psi g=f$ has no solution in ${\mathcal S}({\mathbb R}).$ The case $c=\frac{1}{4}$ is more cumbersome. On one hand, for every subadditive weight $\omega$ and each $0<|\mu|\leq 1$ there is $f\in {\mathcal S}_\omega({\mathbb R})$ such that $f \notin (\mu - C_\psi)({\mathcal S}({\mathbb R})).$ On the other hand, for $|\mu |>1, $  $(\mu -C_\psi)^{-1} \in L({\mathcal S}({\mathbb R}))$ and the series $\displaystyle \sum_{m=0}^\infty \frac{C_{\psi_m}f}{\mu^{m+1}}$ converges in ${\mathcal S}({\mathbb R})$ for every $f \in {\mathcal S}({\mathbb R})$ \cite[Theorem 3.3]{advances}. Next we analyze, for $|\mu|>1,$ the range of $R_\mu:=(\mu -C_\psi)^{-1}$ restricted to spaces ${\mathcal S}_\omega({\mathbb R}).$

\begin{lemma}\label{lem:recurrence} Let $y_0\geq 4$ and $y_{n+1}=\sqrt{2y_n-1}.$ Then $$y_n \geq \left( \frac{n+2}{n+1}\right)^2.$$

\end{lemma}

\begin{proof} We proceed by induction. By hypothesis $y_0\geq 4.$  Suppose that for some $n\geq 1,$   $y_{n-1}\geq \left(\frac{n+1}{n}\right)^2.$ We want to show that $$y_{n}^2=2y_{n-1}-1\geq   \left(\frac{n+2}{n+1}\right)^4.$$ To this end, it suffices to see  that $$2\left(\frac{n+1}{n}\right)^2-1\geq \left(\frac{n+2}{n+1}\right)^4,$$ which happens if and only if $$2(n+1)^6-n^2(n+1)^4\geq n^2(n+2)^4.$$ The left hand side is $$n^6+8n^5+24n^4+36n^3+29n^2+12n+2$$ whereas the right hand side coincides with $$n^6+8n^5+24n^4+32n^3+16n^2.$$ This finishes the proof.

\end{proof}

\begin{theorem}\label{th:neumann-gevrey-secondorder} Let $\psi(x)=x^2+\frac{1}{4},$ $|\mu |>1$ and $R_\mu=\sum_{m=0}^\infty \frac{C_{\psi_m}}{\mu^{m+1}}.$ Then
\begin{enumerate}
\item  Let $1 < d < 2$ be given. For every  strong weight $\omega$ there is $f\in {\mathcal S}_{\omega}({\mathbb R})$ such that $R_\mu f\notin \Sigma_d({\mathbb R}).$

\item Let $d, d' > 1$ be given such that $d' < d + 2.$ There is $f\in \Sigma_d({\mathbb R})$ such that $R_\mu f\notin \Sigma_{d'}({\mathbb R}).$

\end{enumerate}
\end{theorem}

\begin{proof}(1) Take $x_0 \geq 2$ and define $(x_n)_n$, $x_{n}\geq 0,$ by the recurrence rule $x_{n+1}^2+\frac{1}{4}=x_n.$ Then $(x_n)_n$ is a decreasing sequence converging to $\frac{1}{2}.$

For each $n$ let $a_n$ be the scalar sequence (indexed at ${\mathbb N}_0$) $a_n=(\delta_{n,j})_j.$ As the Borel map $B:{\mathcal S}_\omega({\mathbb R}) \to {\mathcal E}_\omega({\mathbb R})$ is surjective and   $(a_n)_n$ is a bounded sequence in the Fréchet nuclear space ${\mathcal E}_{(\omega)}(\{0\})$ (see \cite[Lemma 3.4]{jmaa}), we may find a bounded sequence $(f_n)_n$ in ${\mathcal S}_\omega({\mathbb R})$ such that $(f_n^{(j)}(x_0))_j=a_n.$  Without loss of generality we may assume that all functions $f_n$ have compact support contained in $(x_1, \psi(x_0)),$ which implies that $f_n$ and all their derivatives vanish at $x_j$  as well as at $\psi_j(x_0)$ for $j\geq 1.$

Assume that $R_\mu ({\mathcal S}_{\omega}({\mathbb R}))\subset \Sigma_d({\mathbb R}).$ By the closed graph theorem, the map $R_\mu:{\mathcal S}_{\omega}({\mathbb R})\to \Sigma_d({\mathbb R})$ is continuous,
 thus $(R_\mu(f_n))_n$ is a bounded sequence in $\Sigma_d({\mathbb R}).$ In particular there is $C>0$ such that $$ |(R_\mu f_n)^{(n)}(x_n)|\leq C n!^d\mbox{ for each } n\geq 1.$$ Observe that $\displaystyle (R_\mu f_n)^{(n)}(x_n) = \sum_{j=0}^\infty \frac{1}{\mu^{j+1}}(f_{n} \circ \psi_j)^{(n)}(x_n).$ We denote $y_k:= 2x_k,$ so that $y_{n+1}^2 + 1 = 2y_n.$ Since $\psi_j(x_n)=x_{n-j}$ for $j\leq n,$ using Fa\`a di Bruno's formula and identity $(9)$ in \cite{advances}, we conclude that
$$\begin{array}{*2{>{\displaystyle}l}}

(R_\mu f_n)^{(n)}(x_n) & = \frac{1}{\mu^{n+1}}f_n^{(n)}(x_0)\left(\psi_n'(x_{n})\right)^n  = \frac{1}{\mu^{n+1}}\left(\prod_{j=0}^{n-1}2\psi_j(x_n)\right)^n\\ & \\ & = \frac{1}{\mu^{n+1}}\left(\prod_{k=1}^{n}y_k\right)^n\geq \left(\frac{(n+2)^2}{4\mu}\right)^n\frac{1}{\mu}\end{array}$$ by Lemma \ref{lem:recurrence}. This is a contradiction.

(2) We put $\omega(t) = t^{\frac{1}{d}}$ and $\sigma(t) = t^{\frac{1}{d'}}.$ We take $(x_n)_n$ and $(a_n)_n$ as in (1) and consider $b_n:= \exp\left(\lambda_n\varphi_\omega^\ast(\frac{n}{\lambda_n})\right)a_n,$ where $\lambda_n = \log n.$ Then $(b_n)_n$ is a bounded sequence in ${\mathcal E}_{(\omega)}(\{0\})$ and, again by the surjectivity of the Borel map,  we may find a bounded sequence $(f_n)_n$ in ${\mathcal S}_\omega({\mathbb R})$ such that $(f_n^{(j)}(x_0))_j = b_n.$  Without loss of generality we may assume that all functions $f_n$ have compact support contained in $(x_1, \psi(x_0)),$ which implies that $f_n$ and all their derivatives vanish at $x_j$ as well as at $\psi_j(x_0)$ for $j\geq 1.$

As before, the closed graph theorem implies that $R_\mu:{\mathcal S}_{\omega}({\mathbb R})\to {\mathcal S}_{\sigma}({\mathbb R})$ is continuous provided that $R_\mu ({\mathcal S}_{\omega}({\mathbb R}))\subset {\mathcal S}_{\sigma}({\mathbb R}).$
 Therefore $(R_\mu(f_n))_n$ is a bounded sequence in ${\mathcal S}_{\sigma}({\mathbb R}).$ In particular, there is $C > 0$ such that
$$
|(R_\mu f_n)^{(n)}(x_n)|\leq C \exp\left(\varphi_\sigma^\ast(n)\right)\ \ \forall\ n\in {\mathbb N}_0.$$ On the other hand, proceeding as in the proof of (1) we obtain
$$\left(R_\mu f_n\right)^{(n)}(x_n)\geq \frac{(n+2)^{2n}}{(4\mu)^{n+1}}\exp\left(\lambda_n\varphi_{\omega}^\ast(\frac{n}{\lambda_n})\right).$$ Using that $\varphi_\omega^\ast(x) = xd\log\left(\frac{xd}{e}\right)$ we finally obtain	
$$
\left(\frac{nd}{\lambda_n e}\right)^{nd}\frac{(n+2)^{2n}}{(4\mu)^{n+1}} \leq C \left(\frac{nd'}{e}\right)^{nd'}$$ for all $n\in {\mathbb N}_0,$ which is a contradiction.
\end{proof}

   \begin{remark}{\rm The same argument of the proof of Theorem \ref{th:neumann-gevrey-secondorder}(1) gives that $R_\mu({\mathcal S}_{\omega}({\mathbb R}))\not\subset \Sigma_2({\mathbb R})$ whenever $1 < |\mu| < \frac{e^2}{4}.$}
\end{remark}

\begin{corollary} Let $\omega(t)=\max\{0,\log ^p(t)\}$ with $p>1.$ Given a polynomial $\psi$ of degree greater than one, the following statements hold:\begin{enumerate}
\item $\sigma_{{\mathcal S}_{\omega}({\mathbb R})}(C_\psi)=\{0\}$ whenever $\psi$  lacks fixed points.
\item  $\overline{{\mathbb D}}\setminus \{0\} \subset \sigma_{{\mathcal S}_{\omega}({\mathbb R})}(C_\psi)$ provided that $\psi$ has  fixed points.
\item $\sigma_{{\mathcal S}_{\omega}({\mathbb R})}(C_\psi)={\mathbb C}\setminus \{0\}$ if $\psi$ is of second degree and has two different fixed points.
\end{enumerate}

\end{corollary}

The case of polynomials of degree one is not considered here as, concerning the spectrum, they behave as in the Schwartz class  (see \cite[Proposition 2.1]{advances}).

\par\medskip\noindent
{\bf Acknowledgement.} The research was partially supported by project GV Prometeu/2021/070 and Grant PID2020-119457GB-100, funded by “ERDF A way of making Europe”
 and by MCIN/ AEI/10.13039/501100011033.

\end{document}